\documentclass[12pt]{article}
\usepackage{amsmath,amssymb,amsbsy,amsfonts,amsthm}

\begin{document}

\newtheorem{theorem}{Theorem}
\newtheorem{lemma}[theorem]{Lemma}
\newtheorem{cor}[theorem]{Corollary}
\newtheorem{prop}[theorem]{Proposition}

\newcommand{\comm}[1]{\marginpar{%
\vskip-\baselineskip %raise the marginpar a bit
\raggedright\footnotesize
\itshape\hrule\smallskip#1\par\smallskip\hrule}}

%%%%%%%%%%%%%%%%%%%%%%%%%
% Alphabet calligraphic %
%%%%%%%%%%%%%%%%%%%%%%%%%
\def\cA{{\mathcal A}}
\def\cB{{\mathcal B}}
\def\cC{{\mathcal C}}
\def\cD{{\mathcal D}}
\def\cE{{\mathcal E}}
\def\cF{{\mathcal F}}
\def\cG{{\mathcal G}}
\def\cH{{\mathcal H}}
\def\cI{{\mathcal I}}
\def\cJ{{\mathcal J}}
\def\cK{{\mathcal K}}
\def\cL{{\mathcal L}}
\def\cM{{\mathcal M}}
\def\cN{{\mathcal N}}
\def\cO{{\mathcal O}}
\def\cP{{\mathcal P}}
\def\cQ{{\mathcal Q}}
\def\cR{{\mathcal R}}
\def\cS{{\mathcal S}}
\def\cT{{\mathcal T}}
\def\cU{{\mathcal U}}
\def\cV{{\mathcal V}}
\def\cW{{\mathcal W}}
\def\cX{{\mathcal X}}
\def\cY{{\mathcal Y}}
\def\cZ{{\mathcal Z}}%%

\def\C{\mathbb{C}}
\def\F{\mathbb{F}}
\def\K{\mathbb{K}}
\def\Z{\mathbb{Z}}
\def\R{\mathbb{R}}
\def\Q{\mathbb{Q}}
\def\N{\mathbb{N}}

\def\({\left(}
\def\){\right)}
\def\[{\left[}
\def\]{\right]}
\def\<{\langle}
\def\>{\rangle}

\def\e{e}

\def\eq{\e_q}
\def\eT{\e_T}

\def\fl#1{\left\lfloor#1\right\rfloor}
\def\rf#1{\left\lceil#1\right\rceil}
\def\mand{\qquad\mbox{and}\qquad}

\title{\bf Bounds of Multiplicative Character Sums with Shifted Primes}

\date{ }
\author{
{\sc   Bryce Kerr} \\
{Department of Computing, Macquarie University} \\
{Sydney, NSW 2109, Australia} \\
{\tt  bryce.kerr@mq.edu.au}}

\date{}

\maketitle
 
\begin{abstract}
For integer $q$, let $\chi$ be a primitive multiplicative character$\pmod q.$ For integer $a$ coprime to $q$, we obtain a new bound for the sums
$$\sum_{n\le N}\Lambda(n)\chi(n+a),$$
where $\Lambda(n)$ is the von Mangoldt function. This bound improves and extends the range of a result of Friedlander, Gong and Shparlinski. \end{abstract}


%%%%%%%%%%%%%%%%%%% Main Matter %%%%%%%%%%%%%%%%%%

\section{Introduction}

Let $q$ be an arbitrary positive integer and let $\chi$ be a primitive
non-principal multiplicative character$\pmod q$. Our goal is to estimate character sums of the form
\begin{equation}\label{eqn:vonM}
S_a(q;N)= \sum_{n\leq N}\Lambda(n)\chi(n + a),
\end{equation}
where $a$ is an integer relatively prime to $q$ and as usual,
$$
\Lambda(n)=
\begin{cases}
  \log p,   & \quad\text{if}~n~\text{is a power of a prime}~p, \\
       0,   & \quad\text{otherwise},
\end{cases}
$$
is the von Mangoldt function. For prime modulus $q$, Karatsuba~\cite{Kar} has given
a nontrivial estimate of the sums $S_a(q;N)$ in the range
$N>q^{1/2 + \varepsilon}$. Recently, much more general sums over primes have been considered by Fouvry, Kowalski and Michel~\cite{FKM}. A special case of their general result (see~\cite[Corollary 1.12]{FKM}) gives nontrivial bounds for character sums to prime modulus  $q$, with a very general class of rational functions over primes, which is nontrivial provided $N>q^{3/4+\varepsilon}$. Rakhmonov~\cite{Rakh0,Rakh} has
shown that nontrivial cancellations in
the sums $S_a(q;N)$ also occur in the more difficult case of general
modulus $q$, but only in the narrower range $N>q^{1 + \varepsilon}$.
This range has been extended in~\cite{FGS} to  $N>q^{8/9 + \varepsilon}$,
where the bound
\begin{equation}
\label{FGS1}
|S_a(q;N)| \le (N^{7/8}q^{1/9} +
N^{33/32}q^{-1/18})q^{o(1)},
\end{equation}
is given for  $N\le q^{16/9}$ (since for $N > q^{16/9}$,
the result of~\cite{Rakh} already produces a strong estimate).
Here we give a further improvement.

\begin{theorem}\label{thm:main}
For $N\le q$, we have

$$|S_a(q;N)|\le \left(Nq^{-1/24}+q^{5/42}N^{6/7}\right)q^{o(1)}.$$
\end{theorem}
We see that our new bound is nontrivial for $N>q^{5/6+\varepsilon}$ which gives an improvement over the range $N>q^{8/9+\varepsilon}$ of ~\cite{FGS}. \newline
Our method is based on techniques already established in~\cite{FGS}
however we also introduce some new ideas in the scheme. Our improvement comes from new bounds for the sums (see Lemma~\ref{burgess 2 1} and Lemma~\ref{burgess 3 1})
\begin{equation}
\label{type 1}
\sum_{\substack{M<n\le M+N \\ (n,q)=1}}\chi(n+a),
\end{equation}
and the bilinear forms (see Lemma~\ref{lem:DoubleCharVar4})
\begin{equation}
\label{type 2}
\sum_{k\le K}\sum_{\ell\le L}
\alpha_k\,\beta_\ell\,\chi(k\ell+a).
\end{equation}
To bound the sums~\eqref{type 1} we use standard techniques for dealing with character sums to reduce the problem to bounding the mean value
\begin{equation}
\label{MV111}
\sum_{v_1,\ldots,v_{2r}=1}^V\left|\sum_{x=1}^q\chi\(\prod_{i=1}^r(x+dv_i)\)
\overline{\chi}\(\prod_{i=r+1}^{2r}(x+dv_i)\)\right|, \quad r=2,3,
\end{equation}
which we deal with using techniques of Burgess~\cite{Burg1,Burg3}. To bound the bilinear forms~\eqref{type 2}, we apply the Cauchy-Schwartz inequality, interchange summation and complete the resulting sums. This allows us to reduce the problem to bounding the sums
$$\sum_{\substack{n=1 \\ (n,q)=1}}^{q}\chi\left(1+\frac{b}{n}\right)e^{2 \pi i \lambda n/q},$$
which are dealt with using ideas based on Burgess~\cite{Burg1}. Finally, we note that our new bound for the sums~\eqref{type 1} when combined with the argument from~\cite{FGS} improves on the bound~\eqref{FGS1} although not the range of $N$ for which this bound becomes nontrivial. Our new bound for the  bilinear forms~\eqref{type 2} is what increases the range from $N>q^{8/9+o(1)}$ to $N>q^{5/6+o(1)}$ for which the bound for $|S_a(q;N)|$ becomes nontrivial.
\subsection*{Acknowledgement}
The author would like to thank Igor Shparlinski for his very useful suggestions, particulary in the section on bilinear character sums.
\section{Preliminaries}\label{sec Prelim}
As in~\cite{FGS} our basic tool is the Vaughan identity~\cite{Vau}.
\begin{lemma}\label{lem:Vau} For any complex-valued function $f(n)$ and
any real numbers $U,\, V>1$ with $UV\le N$, we have
$$
S_a(q;N)\ll\Sigma_1+\Sigma_2+\Sigma_3+|\Sigma_4|,
$$
where
\begin{eqnarray*}
\Sigma_1 & = & \left|\sum_{n\le U}\Lambda(n)f(n)\right|, \\
\Sigma_2 & = & (\log UV)\sum_{v \le UV}\left|\sum_{s\le N/v}f(sv)\right|, \\
\Sigma_3 & = & (\log N)\sum_{v\le V}\,\max_{w\ge 1}
\left|\sum_{w\le s\le N/v}f(sv)\right|, \\
\Sigma_4 & = & \sum_{\substack{k\ell\le N\\k>V,\,\ell>U}}\Lambda(\ell)
\sum_{d\mid k,\,d\leq V}\mu(d)f(k\ell).
\end{eqnarray*}
where $\mu(d)$ denotes the M{\"o}bius function, defined by
$$\mu(d)=\begin{cases}(-1)^{\omega(d)}, \quad \text{if} \ \ n  \ \ \text{squarefree}, \\
0, \quad \quad \quad \quad  \text{otherwise}, \quad  \end{cases}$$
and $\omega(d)$ counts the number of distinct prime factors of $d$.
\end{lemma}
\section{P\'{o}lya-Vinogradov Bound}
The following is~\cite[Lemma~4]{FGS}. 
\begin{lemma}
\label{polya vinogradov 1}
For any integers $d,M,N,a$ with $(a,q)=1$ and any primitive character $\chi \pmod q$ we have
$$\left|\sum_{M<n\le M+N}\chi(dn+a)\right|\le (d,q)\frac{N}{q^{1/2}}+q^{1/2+o(1)}.$$
\end{lemma}
Lemma~\ref{polya vinogradov 1} was used to show~\cite[Lemma~5]{FGS}.
\begin{lemma}
\label{polya vinogradov}
For any integers $M,N,a$ with $(a,q)=1$ and any primitive character $\chi \pmod q$ we have
$$\left|\sum_{\substack{M<n\le M+N \\ (n,q)=1}}\chi(n+a)\right|\le q^{1/2+o(1)}+Nq^{-1/2}.$$
\end{lemma}
\section{Burgess Bounds}

In~\cite{FGS}, the Burgess bound for the sums 
\begin{equation*}
\label{burgess sum 1}
\sum_{v_1,\ldots,v_{2r}=1}^V\left|\sum_{x=1}^q\chi\(\prod_{i=1}^r(x+v_i)\)
\overline{\chi}\(\prod_{i=r+1}^{2r}(x+v_i)\)\right| \quad r=2,3,
\end{equation*}
was used to improve on Lemma~\ref{polya vinogradov} for small values of $N$. We a give further improvement by using the methods of Burgess~\cite{Burg1,Burg3} to bound the sums
\begin{equation}
\label{burgess sum 2}
\sum_{v_1,\ldots,v_{2r}=1}^V\left|\sum_{x=1}^q\chi\(\prod_{i=1}^r(x+dv_i)\)
\overline{\chi}\(\prod_{i=r+1}^{2r}(x+dv_i)\)\right|, \quad r=2,3,
\end{equation}
which will then be used with techniques from~\cite{FGS} to obtain new bounds for sums of the form
$$\sum_{\substack{n\le N \\ (n,q)=1}}\chi(n+a).$$
\subsection{The case r=2} 
We use a special case of~\cite[Lemma~7]{Burg1}.
\begin{lemma}
\label{lem:Burgress-L7 1}
For integer $q$ let $\chi$ be a primitive character$\pmod q$ and let 
$$f_1(x)=(x-dv_1)(x-dv_2), \quad f_2(x)=(x-dv_3)(x-dv_4).$$
Suppose at least $3$ of $v_1, v_2, v_3, v_{4}$ are distinct and define
$$A_i=\prod_{j \neq i}(dv_i - dv_j).$$
Then we have
$$\left|\sum_{x=1}^{q}\chi(f_1(x))\overline\chi(f_2(x))\right|\le 8^{\omega(q)}q^{1/2}(q,A_i),$$
for some $A_i\neq 0$ with  $1\le i \le 4$,  where $\omega(q)$ counts the number of distinct prime factors of $q$.
\end{lemma}
We use Lemma~\ref{lem:Burgress-L7 1} and the proof of~\cite[Lemma~8]{Burg1} to show,
\begin{lemma}\label{lem:Burgess-4} For any primitive character $\chi$
modulo $q$ and any positive integer $V$ we have,
$$
\sum_{v_1,\ldots,v_4=1}^V\left|\sum_{x=1}^q\chi\(\prod_{i=1}^2(x+dv_i)\)
\overline{\chi}\(\prod_{i=3}^4(x+dv_i)\)\right|
\le (V^2q+  (d,q)^3q^{1/2}V^4)q^{o(1)}.$$

\end{lemma}
\begin{proof}
We divide the outer summation of  
$$\sum_{v_1,v_2,v_3,v_4=1}^V\left|\sum_{x=1}^q\chi\(\prod_{i=1}^2(x+dv_i)\)
\overline{\chi}\(\prod_{i=3}^4(x+dv_i)\)\right|,$$ into two sets. In the first set we put all $v_1, v_2, v_3, v_4$ which contain at most $2$ distinct numbers and we put the remaining $v_1,v_2,v_3,v_4$ into the second set. The number of elements in the first set is $\ll V^2$ and for these sets we estimate the inner sum trivially. This gives
\begin{align*}
\sum_{v_1,\ldots,v_4=1}^V\left|\sum_{x=1}^q\chi\(\prod_{i=1}^2(x+dv_i)\)
\overline{\chi}\(\prod_{i=3}^4(x+dv_i)\)\right| \ll \\ qV^2+  \sideset{}{'}\sum_{v_1,\ldots,v_4=1}^V\left|\sum_{x=1}^q\chi\(\prod_{i=1}^2(x+dv_i)\)
\overline{\chi}\(\prod_{i=3}^4(x+dv_i)\)\right|,  
\end{align*}
where the last sum is restricted to $v_1, v_2, v_3, v_4$ which contain at least $3$ distinct numbers. With notation as in Lemma~\ref{lem:Burgress-L7 1}, we have
$$\sideset{}{'}\sum_{v_1,\ldots,v_4=1}^V\left|\sum_{x=1}^q\chi\(f_1(x)\)
\overline{\chi}\(f_2(x)\)\right|\le q^{1/2+o(1)}\sideset{}{'}\sum_{v_1,\ldots,v_4=1}^V\sum_{\substack{ i=1 \\ A_i \neq 0}}^{4}(A_i,q).$$
Since $A_i=\prod_{i\neq j}(dv_i-dv_j)=d^3\prod_{i\neq j}(v_i-v_j)=d^3A'_{i},$ we have
$$\sideset{}{'}\sum_{v_1,\ldots,v_4=1}^V\sum_{\substack{ i=1 \\ A_i \neq 0}}^{4}(A_i,q)\le 
(d^3,q)\sideset{}{'}\sum_{v_1,\ldots,v_4=1}^V\sum_{\substack{ i=1 \\ A_i \neq 0}}^{4}(A'_i,q),$$
and in~\cite[Lemma~8]{Burg1} it is shown  
$$\sideset{}{'}\sum_{v_1,\ldots,v_4=1}^V\sum_{\substack{ i=1 \\ A_i \neq 0}}^{4}(A'_i,q)\le V^{4}q^{o(1)},$$
from which the result follows.
\end{proof}
Using Lemma~\ref{lem:Burgess-4} in the proof of~\cite[Lemma 10]{FGS} we get,
\begin{lemma}
\label{burgess 2}
For any primitive character $\chi \pmod q$ and integers $M$, $N$, $a$ and $d$ satisfying
$$N\le q^{5/8}d^{-5/4}, \quad d\le q^{1/6}, \quad (a,q)=1,$$
we have
$$\left|\sum_{M<n\le M+N}\chi(dn+a)\right|\le q^{3/16+o(1)}d^{3/8}N^{1/2}. $$
\end{lemma}
\begin{proof}
We proceed by induction on $N$. Since the result is trivial for $N\le q^{3/8}$, this forms the basis of the induction. We define 
$$U=[0.25Nd^{3/2}q^{-1/4}], \quad V=[0.25d^{-3/2}q^{1/4}],$$
and let
$$\cU=\{ \ 1\le u \le U \ : \ (u,dq)=1 \  \}, \quad  \cV= \{ \ 1\le v \le V \ : \ (v,q)=1 \  \}. $$
By the inductive assumption, for any $\varepsilon>0$  and integer $h\le UV< N$ we have
$$\left|\sum_{M<n\le M+N}\chi(dn+a)\right|\le \left|\sum_{M<n\le M+N}\chi(d(n+h)+a)\right|+ 2q^{3/16+\varepsilon}d^{3/8}h^{1/2},$$
for sufficiently large $q$. Hence
$$\left|\sum_{M<n\le M+N}\chi(dn+a)\right|\le \frac{1}{\#\cU \# \cV}|W|+ 2q^{3/16+\varepsilon}d^{3/8}(UV)^{1/2},$$
where
$$W=\sum_{u\in \cU}\sum_{v \in \cV}\sum_{M<n\le M+N}\chi(d(n+uv)+a)=
\sum_{u\in \cU}\chi(u)\sum_{M<n\le M+N}\sum_{v \in \cV}\chi((dn+a)u^{-1}+dv).
$$
We have
$$|W|\le \sum_{x=1}^{q}\nu(x)\left|\sum_{v \in \cV}\chi(x+dv)\right|,$$
where $\nu(x)$ is the number of representations $x \equiv (dn+a)u^{-1} \pmod q$ with $M<n \le M+N$ and $u \in \cU$.
Two applications of  H\"{o}lder's inequality gives
$$|W|^4\le \left(\sum_{x=1}^{q}\nu^2(x)\right)\left(\sum_{x=1}^{q}\nu(x)\right)^2\sum_{x=1}^{q}\left|\sum_{v \in \cV}\chi(x+dv)\right|^4.$$
From the proof of~\cite[Lemma~7]{FGS} we have
$$\sum_{x=1}^{q}\nu(x)=N\# \cU, \quad  \sum_{x=1}^{q}\nu^2(x)\le \left(\frac{dNU}{q}+1\right)NUq^{o(1)},$$
and by Lemma~\ref{lem:Burgess-4}
\begin{align*}
\sum_{x=1}^{q}\left|\sum_{v \in \cV}\chi(x+dv)\right|^4&=\sum_{v_1,\dots v_4 \in \cV}\sum_{x=1}^{q}
\chi \left(\prod_{i=1}^{2}(x+dv_i)\right)\overline \chi \left(\prod_{i=3}^{4}(x+dv_i) \right) \\
&\le \sum_{v_1,\dots v_4 =1}^{V} \left| \sum_{x=1}^{q}
\chi \left(\prod_{i=1}^{2}(x+dv_i)\right)\overline \chi \left(\prod_{i=3}^{4}(x+dv_i) \right)\right| \\
&\ll V^2q^{1+o(1)},
\end{align*}
 since $V\le d^{-3/2}q^{1/4}$. Combining the above bounds gives
$$|W|^4\le \left(\frac{dNU}{q}+1\right)NU(N\# \cU)^2V^2q^{1+o(1)},$$
and since 
$$\#\cU=Uq^{o(1)}, \quad \#\cV=Vq^{o(1)},$$
we have
\begin{align*}
\left|\sum_{M<n\le M+N}\chi(dn+a)\right| \le 
\left(\frac{d^{1/4}N}{V^{1/2}}+\frac{q^{1/4}N^{3/4}}{U^{1/4}V^{1/2}}\right)q^{o(1)}+2q^{3/16+\varepsilon}d^{3/8}(UV)^{1/2}.
\end{align*}
Recalling the choice of $U$ and $V$ we get
\begin{align*}
\left|\sum_{M<n\le M+N}\chi(dn+a)\right| &\le \left(\frac{dN}{q^{1/8}}+q^{3/16}d^{3/8}N^{1/2}\right)q^{o(1)}+\frac{1}{2}q^{3/16+\varepsilon}d^{3/8}N^{1/2}, \\
\end{align*}
and since by assumption, 
$$N\le q^{5/8}d^{-5/4},$$
we get for sufficiently large $q$
\begin{align*}
\left|\sum_{M<n\le M+N}\chi(dn+a)\right| &\le q^{3/16}d^{3/8}N^{1/2}q^{o(1)}+\frac{1}{2}q^{3/16+\varepsilon}d^{3/8}N^{1/2} \\
 &\le q^{3/16+\varepsilon}d^{3/8}N^{1/2}.
\end{align*}
\end{proof}
\begin{lemma}
\label{burgess 2 1}
Let $\chi$ be a primitive character$\pmod q$ and suppose $(a,q)=1$, then for $N\le q^{43/72}$ we have
$$\left|\sum_{\substack{M<n\le M+N \\ (n,q)=1}}\chi(n+a)\right|\le q^{3/16+o(1)}N^{1/2}.$$
\end{lemma}
\begin{proof}
We have
\begin{align*}
\left|\sum_{\substack{M<n\le M+N \\ (n,q)=1}}\chi(n+a)\right|&=\left|\sum_{d|q}\mu(d)\sum_{M/d<n\le (M+N)/d}\chi(dn+a)\right| \\
&\le \sum_{d|q}\left| \sum_{M/d<n\le (M+N)/d}\chi(dn+a) \right|.
\end{align*}
Let $$Z=\left \lfloor \frac{N^{1/2}}{q^{-3/16}} \right \rfloor,$$
then by Lemma~\ref{burgess 2} we have
\begin{align*}
& \sum_{\substack{d|q \\ d\le Z}}\left| \sum_{M/d<n\le (M+N)/d}\chi(dn+a) \right|= \\ & \quad \quad \quad  \sum_{\substack{d|q\\ d\le Z}}\left| \sum_{M/d<n\le (M+N)/d}\chi(dn+a) \right|+ \sum_{\substack{d|q \\ d>Z}}\left| \sum_{M/d<n\le (M+N)/d}\chi(dn+a) \right| \\
& \quad \quad \quad \quad \quad \quad \quad \quad   \le  \sum_{\substack{d|q\\ d\le Z}}q^{3/16+o(1)}d^{-1/8}N^{1/2}+ \sum_{\substack{d|q\\ d> Z}}\frac{N}{d}.
\end{align*}
By choice of $Z$ we get
\begin{align*}
 \sum_{\substack{d|q\\ d\le Z}}q^{3/16+o(1)}d^{-1/8}N^{1/2}+ \sum_{\substack{d|q\\ d> Z}}\frac{N}{d}&\le \left(q^{3/16}N^{1/2}+\frac{N}{Z} \right)q^{o(1)}
\le q^{3/16+o(1)}N^{1/2},
\end{align*}
which gives the desired bound. It remains to check that the conditions of Lemma~\ref{burgess 2} are satisfied. For each $d|q$ with $d\le Z$, we need
$$\frac{N}{d}\le q^{5/8}d^{-5/4}, \quad d\le q^{1/6}.$$
Recalling the choice of $Z$, we see this is satisfied for $N\le q^{43/72}.$
\end{proof}
\subsection{The case r=3}
Throughout this section we let
\begin{equation}
\label{f definitions}
f_1(x)=(x+dv_1)(x+dv_2)(x+dv_3), \quad f_1(x)=(x+dv_4)(x+dv_5)(x+dv_6),
\end{equation}
and
\begin{equation}
\label{F definitions}
F(x)=f_1'(x)f_2(x)-f_1(x)f_2'(x),
\end{equation}
and write $\mathbf{v}=(v_1,\dots v_6)$. We follow the argument of Burgess~\cite{Burg3} to give an  upper bound for the cardinality of the set
\begin{align*}
\cA(s,s')=\{ & \mathbf{v} \  : \ 0<v_i \le V, \ \text{there exists an $x$ such that} \\ & \ (s,f_1(x)f_2(x))=1, \
s|F(x), \ s|F'(x), \ s'|F''(x) \},
\end{align*}
which will then be combined with the proof of~\cite[Theorem~2]{Burg3} to bound the sums~\eqref{burgess sum 2}.
The proof of the following Lemma is the same as~\cite[Lemma 3]{Burg3}.
\begin{lemma}
\label{A}
Let $s'|s$ and consider the equations
\begin{equation}
\label{cond 1.1}
(\lambda, s)=1, \quad (f_1(-t), s/s')=1,
\end{equation}
\begin{equation}
\label{cond 1.2}
6(f_1(X)+\lambda f_2(X)) \equiv 6(1+\lambda)(X+t)^3 \pmod s,
\end{equation}
\begin{equation}
\label{cond 1.3}
6(1+\lambda) \equiv 0 \pmod {s'}.
\end{equation}
Let
\begin{align*}
\cA_1(s,s')=\{ \mathbf{v},\ & \lambda, \ t \ : \ 0<v_i\le V, \  v_i \neq v_1, \  i \ge 2,  \\  & \ 0<\lambda \le s, \  0<t \le s/s', \ \eqref{cond 1.1},\ \eqref{cond 1.2}, \ \eqref{cond 1.3} \},
\end{align*}
then we have
$$\# \cA(s,s') \ll V^3+\# \cA_1(s,s').$$
\end{lemma}
We next make the substitutions
\begin{align}
\label{substitutions}
Y&=X+dv_1, \nonumber \\
V_i&=v_i-v_1, \quad i\ge 2, \\ 
T&= t-dv_1 \pmod {s/s'} , \nonumber
\end{align}
so that 
\begin{align}
\label{f sub 1}
f_1(X)&=Y(Y+dV_2)(Y+dV_3)=Y^3+d(V_2+V_3)Y^2+d^2V_2V_3Y \nonumber \\ &=g_1(Y),
\end{align}
\begin{align}
\label{f sub 2}
f_2(X)&=(Y+dV_4)(Y+dV_5)(Y+dV_6)=Y^3+d\sigma_1 Y^2+d^2\sigma_2Y+d^3\sigma_3 \nonumber \\&= g_2(Y),
\end{align}
where
\begin{align}
\label{sigma}
\sigma_1&=V_4+V_5+V_6, \nonumber \\
\sigma_2&=V_4V_5+V_4V_6+V_5V_6, \\
\sigma_3&=V_4V_5V_6, \nonumber
\end{align}
and we see that~\eqref{cond 1.2} becomes
\begin{equation}
\label{cond 1.2.1}
6(g_1(Y)+\lambda g_2(Y)) \equiv 6(1+\lambda)(Y+T)^3 \pmod s.
\end{equation}
The proof of the following Lemma follows that of~\cite[Lemma~4]{Burg3}. 
\begin{lemma}
\label{A1}
With notation as in~\eqref{substitutions} and~\eqref{sigma}, consider the equations
\begin{equation}
\label{cond 2.1}
(s/s',T)=1, \quad (s/s',T-dM_3)=1,
\end{equation}
\begin{equation}
\label{cond 2.2}
6d^2T^3(V_3^2-\sigma_1V_3+\sigma_2)-18d^3\sigma_3T^2+18d^4V_3\sigma_3T-6d^5V_3^2\sigma_3 \equiv 0 \pmod s,
\end{equation}
\begin{equation}
\label{cond 2.3}
6d^3\sigma_3 \equiv 0 \pmod {s'},
\end{equation}
and let
\begin{align*}
\cA_2(s,s')=\{ (V_3, & \ V_4, \ V_5, \ V_6, \ T) \ : \\ & \ 0<|M_i|\le V, \  0<T\le s/s', \ \eqref{cond 2.1},\ \eqref{cond 2.2}, \ \eqref{cond 2.3} \}.
\end{align*}
Then we have
$$ \# \cA_1(s,s')\ll (d,s)V(1+V/q) \# \cA_2(s,s'). $$
\end{lemma}
\begin{proof}
We first note that~\eqref{cond 1.1} and~\eqref{substitutions} imply~\eqref{cond 2.1}. Let
\begin{align*}
\cB_1=\{ (V_2,V_3, & V_4,V_5,V_6,T) \ : \ 0<|V_i|\le V, \\ & \ 0<\lambda \le s, \ (\lambda,s)=1, \ 0<T \le s/s', \eqref{cond 1.3},  \eqref{cond 1.2.1}, \eqref{cond 2.1} \},
\end{align*}
 so that
$$\# \cA_1(s,s') \le V \# \cB_1.$$
Using~\eqref{f sub 1} and~\eqref{f sub 2} and considering common powers of $Y$ in~\eqref{cond 1.2.1} we get
\begin{equation}
\label{cong 1}
6d(V_2+V_3+\lambda \sigma_1)\equiv 18(1+\lambda)T \pmod s, 
\end{equation}
\begin{equation}
\label{cong 2}
6d^2(v_2V_3+\lambda \sigma_2) \equiv 18(1+\lambda)T^2 \pmod s,
\end{equation}
\begin{equation}
\label{cong 3}
6d^3\lambda \sigma_3 \equiv  6(1+\lambda)T^3 \pmod s.
\end{equation}
By~\eqref{cong 1} we see that
$$6dV_2\equiv 18(1+\lambda)T-dV_3-d\lambda \sigma_1 \pmod s,$$
which has $O\left((d,s)(1+V/q)\right)$ solutions in $V_2$. 
The equations~\eqref{cond 1.3} and~\eqref{cong 3} imply that 
$$6d^3\sigma_3 \equiv 0 \pmod {s'},$$
$$6(1+\lambda) \equiv 0 \pmod {s'},$$
and
\begin{equation}
\label{cong 4}
6\lambda(d^3\sigma_3-T^3) \equiv  6T^3 \pmod s.
\end{equation}
Since $(T,s/s')=1$ by the above equations, there are $O(1)$ possible values of $\lambda$. Finally 
 combining
\eqref{cong 1}, \eqref{cong 2} and \eqref{cong 4} gives~\eqref{cond 2.2}.
\end{proof}
The following is~\cite[Lemma~2]{Burg3}.
\begin{lemma}
\label{burgess  L2}
For any integer $s$ and  polynomial $G(X)$ with integer coefficients, we have
$$\# \{   0\le x<s, \ G(x)\equiv 0 \pmod s, \  (s, G'(x))|6 \}\le s^{o(1)},$$
where the term $o(1)$ depends only on the degree of $G$.
\end{lemma}
The proof of the following Lemma follows that of~\cite[Lemma~5]{Burg3}.
\begin{lemma}
\label{A2}
For $s'' |( s/s')$ consider the equations
\begin{equation}
\label{cond 3.1}
(s , 6d^3\sigma_3)=s' s'',
\end{equation}
\begin{equation}
\label{cond 3.2}
6d^2(V_3^2-\sigma_1V_3+\sigma_2)\equiv 0 \pmod {s},
\end{equation}
and let
$$\cA_3(s,s',s'')= \{ (V_3, V_4, V_5, V_6) \  : \ 0<|V_i|\le V, \ \eqref{cond 3.1}, \ \eqref{cond 3.2} \}.$$
 Then we have
$$ \# \cA_2(s,s')\le s^{o(1)}\sum_{s''|s/s'}s'' \# \cA_3(s,s',s'').$$
\end{lemma}
\begin{proof}
For $s''|(s/s')$, let
$$\cA'_3(s,s',s'')= \{ \ (V_3, V_4, V_5, V_6,T)\in \cA_2(s,s') \  : \ (s,6d^3\sigma_3)=s's'' \},$$
so that
\begin{equation}
\label{decomposition 1}
\#\cA_2(s,s')=\sum_{s''|(s/s')} \# \cA'_3(s,s',s'').
\end{equation}
Let $S=(s',s/s')$, so that $(s'/S, s/s')=1$.
For elements of $\cA_3(s,s',s'')$, since
\begin{equation}
6d^3\sigma_3\equiv 0 \pmod {Ss''},
\end{equation}
 we have by~\eqref{cond 2.1},~\eqref{cond 2.2} and~\eqref{cond 3.1}
\begin{equation}
\label{A_3 equiv}
6d^2(V_3^2-\sigma_1V_3+\sigma_2)\equiv 0 \pmod {Ss''},
\end{equation} 
hence~\eqref{cond 2.2} implies that
\begin{align}
\label{eqn G}
& \frac{6d^2(V_3^2-  \sigma_1V_3+\sigma_2)}{Ss''}T^3-\frac{18d^3\sigma_3}{Ss''}T^2 \nonumber \\ & \quad \quad
+\frac{18d^4\sigma_3V_3}{Ss''}T-\frac{6d^5\sigma_3V_3^2}{Ss''}\equiv 0 \pmod {s/(s's'')}.
\end{align}
Let
\begin{equation*}
G(T)= \frac{6d^2(V_3^2-  \sigma_1V_3+\sigma_2)}{Ss''}T^3-\frac{18d^3\sigma_3}{Ss''}T^2
+\frac{18d^4\sigma_3V_3}{Ss''}T-\frac{6d^5\sigma_3V_3^2}{Ss''},
\end{equation*}
so that
$$3G(T)-TG'(T)=-\frac{18d^3\sigma_3}{Ss''}(T-dV_3)^2.$$ 
Writing \ $6d^3\sigma_3=s's''\sigma'$ with $(\sigma',s)=1$, we see from~\eqref{cond 2.1} that for some integer $y$ with $(y,s/s'
)=1$ that
$$3G(T)-TG'(T)=-\frac{3s'}{S}y.$$ 
If $T_0$ is a root of $G(T) \pmod{s/(s's'')}$  then since $(s'/S,s/s')=1$ we have
$$(G'(T_0), s/(s's''))|6,$$
hence from Lemma~\ref{burgess  L2}, the number of possible values for $T$ is $\ll s''s^{o(1)}$. Finally~\eqref{A_3 equiv} implies 
$$6d^2(V_3^2-\sigma_1V_3+\sigma_2)\equiv 0 \pmod {s''},$$
and the result follows from~\eqref{decomposition 1}.
\end{proof}
\begin{lemma}
With notation as in Lemma~\ref{A2}, for integers $s,s',s''$ satisfying $s'|s$ and $s''|s/s'$ we have
\label{A3}
$$\# \cA_3(s,s',s'') \le  (d^3,s)V^4s^{o(1)}/(s's''). $$
\end{lemma}
\begin{proof}
Bounding the number of solutions to the equation~\eqref{cond 3.2} trivially and recalling the definition of $\sigma_3$ from~\eqref{sigma}, we see that
\begin{align}
\label{ub a3}
\# \cA_3(s,s',s'')\le  V\# \{  (V_4, V_5, V_6) \  : \ 0<|M_i|\le V, \  (s,6d^3V_4V_5V_6)=s's''\}.
\end{align}
Writing $s=(d^3,s)s_1$, $d^3=(d^3,s)d_1$, we see that
$$(s,6d^3V_4V_5V_6)=s's'',$$
 implies
$$(s_1,6V_4V_5V_6)=s's''/(d^3,s).$$
For integers $s_1,s_2,s_3$, let
$$\cA_4(s_1,s_2,s_3)= \{  V_4, V_5, V_6 \  : \ 0<|V_i|\le V,  \ s_1| 6V_4, \  s_2| 6V_5, \  s_3| 6V_6 \},$$
so that from~\eqref{ub a3}
$$\# \cA_3(s,s',s'')\le  V\sum_{s_1s_2s_3=s's''/(d^3,s)}\cA_4(s_1,s_2,s_3).$$
Since 
$$\cA_4(s_1,s_2,s_3)\ll \frac{V^3}{s_1s_2s_3}=\frac{(d^3,s)V^3}{s's''},$$
we see that
$$\# \cA_3(s,s',s'')\le  \frac{(d^3,s)V^4s^{o(1)}}{s's''}.$$
\end{proof}
Combining the above results we get
\begin{lemma}
\label{T1 B}
Let $s' | s$ and 
\begin{align*}
\cA(s,s')=\{ \ & d\mathbf{v} \  : \ 0<v_i \le V, \text{there exists an x such that} \\ & \ (s,f_1(x)f_2(x))=1, \
s|F(x), \ s|F'(x), \ s'|F''(x) \}.
\end{align*}
Then
$$\# \cA(s,s')\le (d,s)^4\left(\frac{V^6}{ss'}+\frac{V^5}{s'}\right)q^{o(1)}+V^3. $$
\end{lemma}
\begin{proof}
From Lemma~\ref{A}, Lemma~\ref{A1}, Lemma~\ref{A2} we see that
\begin{align*}
\# \cA(s,s')\le V^3+(d,s)\left(1+\frac{V}{s}\right)V\sum_{s''|s/s'}s''\#\cA_3(s,s',s''),
\end{align*}
and from Lemma~\ref{A3} we have
\begin{align*}
\sum_{s''|s/s'}s''\#\cA_3(s,s',s'')\le \frac{s^{o(1)}(d^3,s)V^4}{s'},
\end{align*}
which gives the desired result.
\end{proof}
For integer $q$, we define the numbers $h_1(q),h_2(q),h_3(q)$ as in~\cite{Burg3}, 
\begin{align}
\label{h def}
h_1(q)^2&=\text{smallest square divisible by } q, \nonumber \\
h_2(q)^3&= \text{smallest cube divisible by } q, \\
h_3(q)&= \text{ product of distinct prime factors of } q. \nonumber
\end{align}
The following is~\cite[Theorem~2]{Burg3}.
\begin{lemma}
\label{T2 B}
Let $\chi$ be a primitive character mod $q$ and let
\begin{align}
\label{q definition}
q=q_0q_1q_2q_3,
\end{align}
where the $q_i$ are pairwise coprime. Let the integers $l_0,l_1,l_2$ satisfy
\begin{align}
\label{l definition}
l_0 | h_1(q_0)/h_3(q_0), \quad l_1 | h_2(q_1)/h_3(q_1), \quad l_2 | h_2(q_2)/h_3(q_2),
\end{align}
and consider the equations
\begin{equation}
\label{cond 4.1}
l_0h_1(q_1q_2q_3)|F(x), \quad (F(x),h_1(q_0))=l_0,
\end{equation}
\begin{equation}
\label{cond 4.2}
l_1h_2(q_2q_3)|F'(x), \quad (F'(x),h_2(q_1))=l_1,
\end{equation}
\begin{equation}
\label{cond 4.3}
l_2h_2(q_3)|F''(x), \quad (F''(x),h_2(q_2))=l_2,
\end{equation}
and let 
$$\cC=\cC(l_0,l_1,l_2,q_0,q_1,q_2,q_3)= \{ 1\le x \le q \ : \eqref{cond 4.1},~\eqref{cond 4.2},~\eqref{cond 4.3} \}.$$
Then we have
$$\left| \sum_{x \in \cC}\chi(f_1(x))\overline\chi(f_2(x))\right|\le q^{1/2+o(1)}\frac{(q_2q_3l_1)^{1/2}l_2}{h_2(q_2)}.$$
\end{lemma}
\begin{lemma}\label{lem:Burgess-6} For any primitive character $\chi$
modulo $q$ and any
integer $V<q^{1/6}d^{-2}$, we have
$$
\sum_{v_1,\ldots,v_6=1}^V\left|\sum_{x=1}^q\chi\(\prod_{i=1}^3(x+dv_i)\)
\overline{\chi}\(\prod_{i=4}^6(x+dv_i)\)\right| \le V^3 q^{1+o(1)}.
$$
\end{lemma}
\begin{proof}
With notation as above and in~\eqref{f definitions} and~\eqref{F definitions}
$$\sum_{v_1,\ldots,v_6=1}^V\left|\sum_{x=1}^q\chi\(\prod_{i=1}^3(x+dv_i)\)
\overline{\chi}\(\prod_{i=4}^6(x+dv_i)\)\right| \le \sum_{d_i,l_i}\left| \sum_{x \in \cC}\chi(f_1(x))\overline\chi(f_2(x))\right|,$$
where the last sum is extended over all $q_1, q_2, q_3, q_4$ and $l_0,l_1,l_2$ satisfying the conditions of Lemma~\ref{T2 B}. Hence by Lemma~\ref{T1 B}, for some fixed $q_1,\dots,q_4$ satisfying~\eqref{q definition} and $l_0,l_1,l_2$ satisfying~\eqref{l definition},
\begin{align*}
&\sum_{v_1,\ldots,v_6=1}^V\left|\sum_{x=1}^q\chi\(\prod_{i=1}^3(x+dv_i)\)
\overline{\chi}\(\prod_{i=4}^6(x+dv_i)\)\right|\le \\ & \quad  \left(  (q,d)^4\left(\frac{V^6}{l_1h_2(q_2q_3)l_2h_2(q_3)}+\frac{V^5}{l_2h_2(q_3)}\right)q+V^3 \right)\frac{(qq_2q_3l_1)^{1/2}l_2}{h_2(q_2)}q^{o(1)} 
\le \\ & \quad \left( (q,d)^4V^6q^{1/2}+ (q,d)^4V^5q^{2/3}+V^3q\right)q^{o(1)},
\end{align*}
from the definition of $l_i, h_i,q_i$. The result follows since the term $V^3q$ dominates for $V\le q^{1/6}d^{-2}$.
\end{proof}
\begin{lemma}
\label{burgess 3}
For any primitive character $\chi$ modulo $q$ and integers $M$, $N$, $d$ and $a$ satisfying 
 $$N\le q^{7/12}d^{-3/2} , \quad d\le q^{1/12} , \quad (a,q)=1,$$ we have
$$\left|\sum_{M<n\le M+N}\chi(dn+a)\right|\le q^{1/9+o(1)}d^{2/3}N^{2/3}.$$
\end{lemma}
\begin{proof}
Using the same argument from Lemma~\ref{burgess 2}, we proceed by induction on $N$. Since the result is trivial for $N\le q^{1/3}$, this forms the basis of our induciton. Define 
$$U=[0.5Nd^{2}q^{-1/6}], \quad V=[0.5d^{-2}q^{1/6}],$$
and let
$$\cU=\{ \ 1\le u \le U \ : \ (u,dq)=1 \  \}, \quad  \cV= \{ \ 1\le v \le V \ : \ (v,q)=1 \  \}. $$
Fix $\varepsilon>0$, by the inductive hypothesis, for any integer $h\le UV< N$ we have
$$\left|\sum_{M<n\le M+N}\chi(dn+a)\right|\le \left|\sum_{M<n\le M+N}\chi(d(n+h)+a)\right|+ 2q^{1/9+\varepsilon}d^{2/3}h^{2/3},$$
for sufficiently large $q$. Hence
$$\left|\sum_{M<n\le M+N}\chi(dn+a)\right|\le \frac{1}{\#\cU \# \cV}|W|+2q^{1/9+\varepsilon}d^{2/3}(UV)^{2/3},$$
where
$$W=\sum_{u\in \cU}\sum_{v \in \cV}\sum_{M<n\le M+N}\chi(d(n+uv)+a)=
\sum_{u\in \cU}\chi(u)\sum_{M<n\le M+N}\sum_{v \in \cV}\chi((dn+a)u^{-1}+dv).
$$
We have
$$|W|\le \sum_{x=1}^{q}\nu(x)\left|\sum_{v \in \cV}\chi(x+dv)\right|,$$
where $\nu(x)$ is the number of representations $x \equiv (dn+a)u^{-1} \pmod q$ with $M<n \le M+N$ and $u \in \cU$.
Two applications of H\"{o}lder's inequality gives,
$$|W|^6\le \left(\sum_{x=1}^{q}\nu^2(x)\right)\left(\sum_{x=1}^{q}\nu(x)\right)^4\sum_{x=1}^{q}\left|\sum_{v \in \cV}\chi(x+dv)\right|^6.$$
As in Lemma~\ref{burgess 2}
$$\sum_{x=1}^{q}\nu(x)=N\# \cU, \quad  \sum_{x=1}^{q}\nu^2(x)\le \left(\frac{dNU}{q}+1\right)NUq^{o(1)},$$
and by Lemma~\ref{lem:Burgess-6}
\begin{align*}
\sum_{x=1}^{q}\left|\sum_{v \in \cV}\chi(x+dv)\right|^3&=\sum_{v_1,\dots v_4 \in \cV}\sum_{x=1}^{q}
\chi \left(\prod_{i=1}^{3}(x+dv_i)\right)\overline \chi \left(\prod_{i=4}^{6}(x+dv_i) \right) \\
&\le \sum_{v_1,\dots v_4 =1}^{V} \left| \sum_{x=1}^{q}
\chi \left(\prod_{i=1}^{2}(x+dv_i)\right)\overline \chi \left(\prod_{i=3}^{4}(x+dv_i) \right)\right| \\
&\le V^3q^{1+o(1)}.
\end{align*}
The above bounds give
$$|W|^6\le \left(\frac{dNU}{q}+1\right)NUq^{o(1)}(N\# \cU)^4\left(V^3q\right)q^{o(1)},$$
so that
$$\left|\sum_{M<n\le M+N}\chi(dn+a)\right|\le \left(\frac{d^{1/6}N}{V^{1/2}}+ \frac{q^{1/6}N^{5/6}}{U^{1/6}V^{1/2}}\right)q^{o(1)}+2q^{1/9+\varepsilon}d^{2/3}(UV)^{2/3}.$$
Recalling  the choice of $U$ and $V$ we get 
$$\left|\sum_{M<n\le M+N}\chi(dn+a)\right|\le \frac{d^{7/6}N}{q^{1/12+o(1)}}+q^{1/9+o(1)}d^{2/3}N^{2/3}
+\frac{2}{5}q^{1/9+\varepsilon}d^{2/3}N^{2/3},
$$
and since 
$$\frac{d^{7/6}N}{q^{1/12}}\le q^{1/9}d^{2/3}N^{2/3} \quad \text{when} \quad  dN\le q^{13/24},$$
we have by assumption on $N$ and $d$
\begin{align*}
\left|\sum_{M<n\le M+N}\chi(dn+a)\right|&\le q^{1/9+o(1)}d^{2/3}N^{2/3}
+\frac{2}{5}q^{1/9+\varepsilon}d^{2/3}N^{2/3} \\
&\le q^{1/9+\varepsilon}d^{2/3}N^{2/3}, 
\end{align*}
for sufficiently large $q$.
\end{proof}
Using Lemma~\ref{burgess 3} as in the proof of Lemma~\ref{burgess 2 1} we get,
\begin{lemma}
\label{burgess 3 1}
Let $\chi$ be a primitive character$\pmod q$ and suppose $(a,q)=1$, then for $N\le q^{23/42}$ we have
$$\left|\sum_{\substack{M<n\le M+N \\ (n,q)=1}}\chi(n+a)\right|\le q^{1/9+o(1)}N^{2/3}.$$
\end{lemma}
\begin{proof}
We have
\begin{align*}
\left|\sum_{\substack{M<n\le M+N \\ (n,q)=1}}\chi(n+a)\right|&=\left|\sum_{d|q}\mu(d)\sum_{M/d<n\le (M+N)/d}\chi(dn+a)\right| \\
&\le \sum_{d|q}\left| \sum_{M/d<n\le (M+N)/d}\chi(dn+a) \right|.
\end{align*}
Let $$Z=\left\lfloor\frac{N^{1/3}}{q^{-1/9}} \right \rfloor,$$
then by Lemma~\ref{burgess 2} we have
\begin{align*}
& \sum_{\substack{d|q \\ d\le Z}}\left| \sum_{M/d<n\le (M+N)/d}\chi(dn+a) \right|= \\ & \quad \quad \quad  \sum_{\substack{d|q\\ d\le Z}}\left| \sum_{M/d<n\le (M+N)/d}\chi(dn+a) \right|+ \sum_{\substack{d|q \\ d>Z}}\left| \sum_{M/d<n\le (M+N)/d}\chi(dn+a) \right| \\
& \quad \quad \quad \quad \quad \quad \quad \quad   \le  \sum_{\substack{d|q\\ d\le Z}}q^{1/9+o(1)}N^{2/3}+ \sum_{\substack{d|q\\ d> Z}}\frac{N}{d}.
\end{align*}
Since by choice of $Z$
\begin{align*}
 \sum_{\substack{d|q\\ d\le Z}}q^{1/9+o(1)}N^{2/3}+ \sum_{\substack{d|q\\ d> Z}}\frac{N}{d}&\le \left(q^{1/9}N^{2/3}+\frac{N}{Z} \right)q^{o(1)}
\le q^{1/9+o(1)}N^{2/3},
\end{align*}
we get the desired bound. It remains to check that the conditions of Lemma~\ref{burgess 2} are satisfied. For each $d|q$ with $d\le Z$ we need
$$\frac{N}{d}\le q^{7/12}d^{-3/2}, \quad d\le q^{1/12},$$
and from the choice of $Z$, this is satisfied for $N\le q^{23/42}.$
\end{proof}
\section{Bilinear Character Sums}
\begin{lemma}
\label{trans sum}
Let $\chi$ be a primitive character$\pmod q.$ Then for integers $u_1,u_2, \lambda$ we have
$$\left|\sum_{n=1}^{q}\chi(n+u_1)\overline \chi(n+u_2)e^{2\pi i \lambda n/q}\right|=
\left|\sum_{n=1}^{q}\chi(n+\lambda)\overline \chi(n)e^{2\pi i (u_1-u_2) n/q}\right|.
 $$
\end{lemma}
\begin{proof}
Let $$\tau(\chi)=\sum_{n=1}^{q}\chi(n)e^{2\pi i n /q},$$ 
be the Gauss sum, so that 
$$\left|\tau(\chi)\right|=q^{1/2} \quad \text{and} \quad \sum_{n=1}^{q}\chi(n)e^{2\pi i an /q}=\overline\chi(a)\tau(\chi).$$
Writing
$$\chi(n+u_1)=\frac{1}{\tau(\overline\chi)}\sum_{\lambda_1=1}^{q}\overline\chi(\lambda_1)e^{2\pi i (n+u_1)\lambda_1 /q},$$
and
$$\overline\chi(n+u_1)=\frac{1}{\tau(\chi)}\sum_{\lambda_2=1}^{q}\chi(\lambda_2)e^{2\pi i (n+u_2)\lambda_2 /q},$$
we have 
\begin{align*}
& \sum_{n=1}^{q}\chi(n+u_1)\overline \chi(n+u_2)e^{2\pi i \lambda n/q}= \\ & \frac{1}{\tau(\chi)\tau(\overline \chi)} \sum_{\lambda_1=1}^{q}\sum_{\lambda_2=1}^{q}\overline \chi(\lambda_1)e^{2\pi i \lambda_1 u_2/q} \chi(\lambda_2)e^{2\pi i \lambda_2 u_2/q}\sum_{n=1}^{q}e^{2\pi i n(\lambda+\lambda_1+\lambda_2)},
\end{align*}
since
\begin{align*}
& \sum_{\lambda_1=1}^{q}\sum_{\lambda_2=1}^{q}\overline \chi(\lambda_1)e^{2\pi i \lambda_1 u_2/q} \chi(\lambda_2)e^{2\pi i \lambda_2 u_2/q}\sum_{n=1}^{q}e^{2\pi i n(\lambda+\lambda_1+\lambda_2)}= \\ &
\chi(-1)e^{-2\pi i u_2\lambda/q}q\sum_{\lambda_1=1}^{q}\chi(\lambda_1+\lambda)\overline \chi(\lambda_1)e^{2\pi i\lambda_1(u_1-u_2)/q},
\end{align*}
we have
\begin{align*}
\left| \sum_{n=1}^{q}\chi(n+u_1)\overline \chi(n+u_2)e^{2\pi i \lambda n/q} \right|&=\frac{q}{|\tau(\chi)|}\left|\sum_{n=1}^{q}\chi(n+\lambda)\overline \chi(n)e^{2\pi i n(u_1-u_2)/q}\right| \\
&=\left|\sum_{n=1}^{q}\chi(n+\lambda)\overline \chi(n)e^{2\pi i n(u_1-u_2)/q}\right|.
\end{align*}
\end{proof}
\begin{lemma}
\label{character sum bound 11}
Let $\chi$ be a primitive character$\pmod q$, then for integers $b, \lambda$ with $b \not \equiv 0 \pmod q$ we have
$$\left| \sum_{\substack{n=1 \\ (n,q)=1}}^{q}\chi\left(1+\frac{b}{n}\right)e^{2 \pi i \lambda n/q}\right|\le (b,q)q^{1/2+o(1)}. $$
\end{lemma}
\begin{proof}
Consider first when $\lambda \equiv 0 \pmod q$. Then from Lemma~\ref{trans sum} we have
\begin{align*}
\left| \sum_{\substack{n=1 \\ (n,q)=1}}^{q}\chi\left(1+\frac{b}{n}\right)e^{2 \pi i \lambda n/q}\right|&=
\left|\sum_{n=1}^{q}|\chi(n)|e^{2\pi i b n/q}\right|=\left|\sum_{\substack{n=1 \\ (n,q)=1}}^{q}e^{2\pi i b n/q}\right|,
\end{align*}
and from~\cite[Equation~3.5]{IwKow} we have
$$\left|\sum_{\substack{n=1 \\ (n,q)=1}}^{q}e^{2\pi i b n/q}\right|\ll (b,q),$$
so that
$$\left| \sum_{\substack{n=1 \\ (n,q)=1}}^{q}\chi\left(1+\frac{b}{n}\right)e^{2 \pi i \lambda n/q}\right|\ll (b,q)\le (b,q)q^{1/2+o(1)}.$$
Next consider when $\lambda \not \equiv 0 \pmod q$. We first note that if $\chi$ is a character$\pmod p,$ with $p$ prime, then we have from the Weil bound, see~\cite[Theorem 2G]{Schm}
$$\left| \sum_{\substack{n=1 \\ (n,p)=1}}^{p}\chi\left(1+\frac{b}{n}\right)e^{2 \pi i \lambda n/p}\right|\ll p^{1/2}.$$
For $p$ prime and integers $\lambda,b,c,\alpha$, let
$N(\lambda,b, c,p^{\alpha})$ denote the number of solutions to the congruence
\begin{equation}
\label{quad congruence}
\lambda n^2\equiv cb \pmod {p^{\alpha}}, \quad  1\le n \le p^{\alpha}, \quad (n,p)=1.
\end{equation}
Then we have
\begin{equation}
\label{cong bound}
N(\lambda,b, c,p^{\alpha})\le 4(\lambda,p^{\alpha}),
\end{equation}
since if there exists a solution $n$ to~\eqref{quad congruence} then we must have $(\lambda,p^{\alpha})=(cd,p^{\alpha})$ so that we arrive at the congruence
\begin{equation}
\label{aaaaa} 
n^2\equiv a \pmod {p^{\alpha}/(p^{\alpha},\lambda)},
\end{equation}
for some integer $a$ with $(a,p)=1$. Since there are at most $4$ solutions to~\eqref{aaaaa} we get~\eqref{cong bound}.
Suppose $q=p^{2\alpha}$ is an even prime power and let $c$ be defined by
\begin{equation*}
\label{chi add}
\chi(1+p^{\alpha})=e^{2\pi i c/p^{\alpha}},
\end{equation*}
then from the argument of~\cite[Lemma~2]{Burg1}  (see also~\cite[Lemma 12.2]{IwKow}) we have by~\eqref{cong bound}
\begin{align*}
\left| \sum_{\substack{n=1 \\ (n,q)=1}}^{q}\chi\left(1+\frac{b}{n}\right)e^{2 \pi i \lambda n/q}\right| &\ll  p^{\alpha}N(\lambda,b,c) \ll(\lambda,q)q^{1/2}.
\end{align*}
Suppose next $q=p^{2\alpha+1}$ is an odd prime power, with $p>2$. Let $c$ be defined by
$$\chi(1+p^{\alpha+1})=e^{2\pi i c/p^{\alpha}},$$
then from the argument of~\cite[Lemma~4]{Burg1} (see also~\cite[Lemma~12.3]{IwKow})
\begin{align*}
\left| \sum_{\substack{n=1 \\ (n,q)=1}}^{q}\chi\left(1+\frac{b}{n}\right)e^{2 \pi i \lambda n/q}\right| &\ll 
p^{(2\alpha+1)/2}N(\lambda,b, c,p^{\alpha})+p^{\alpha}N(\lambda,b, c,p^{\alpha+1}) \\ &\ll (\lambda,q)q^{1/2}.
\end{align*}
Finally if $q=2^{2\alpha+1}$, then from the argument of~\cite[Lemma~3]{Burg1}
\begin{align*}
\left| \sum_{\substack{n=1 \\ (n,q)=1}}^{q}\chi\left(1+\frac{b}{n}\right)e^{2 \pi i \lambda n/q}\right|&\ll 2^{1/2}2^{\alpha}N(\lambda,b, c,p^{\alpha}) \\ &\ll  (\lambda,q)q^{1/2}.
\end{align*}
Combining the above bounds gives the desired result when $q$ is a prime power. For the general case, suppose $\chi$ is a primitive character$\pmod q$ and let $q=p_1^{\alpha_1}p_2^{\alpha_2}...p_k^{\alpha_k}$ be the prime factorization of $q$. By the Chinese Remainder Theorem we have
$$\chi=\chi_1\chi_2...\chi_k,$$
where each $\chi_i$ is a primitive character$\pmod {p_i^{\alpha_i}}$. Let $q_i=q/p^{\alpha_i}$, then by the above bounds and another application of the Chinese remainder theorem (see~\cite[Equation 12.21]{IwKow}), for some absolute constant $C$
\begin{align*}
& \left|\sum_{\substack{n=1 \\ (n,q)=1}}^{q}\chi\left(1+\frac{b}{n}\right)e^{2 \pi i \lambda n/q}\right|=\\ &\left|\sum_{\substack{n_1=1 \\ (n_1,p_1)=1}}^{p_1^{\alpha_1}}\dots \sum_{\substack{n_k=1 \\ (n_k,p_k)=1}}^{p_k^{\alpha_k}}\chi_1\left(1+\frac{b}{\sum_{i=1}^{k}n_iq_i}\right)e^{2\pi i \lambda n_1/p_1^{\alpha_1}}...
\chi_k\left(1+\frac{b}{\sum_{i=1}^{k}n_iq_i}\right)e^{2 \pi i \lambda n_k/p_i^{\alpha_k}} \right| \\
& = \left|\prod_{i=1}^{k}\left(\sum_{\substack{n_i=1 \\ (n_i,p_i)=1}}^{p_i^{\alpha_i}}\chi_i\left(1+\frac{b}{n_iq_i}\right)e^{2\pi i \lambda n_i/p_i^{\alpha_i}}\right)\right|\le \prod_{i=1}^{k}C(\lambda,p_i^{\alpha_i})p_i^{\alpha_i/2}\le (\lambda,q)q^{1/2+o(1)},
\end{align*}
and the result follows from Lemma~\ref{trans sum}.
\end{proof}
\begin{lemma}
\label{lem:DoubleChar2}
 Let $K, L$ be natural numbers and
for any two sequences $(\alpha_k)_{k=1}^K$
and $(\beta_\ell)_{\ell=1}^L$ of complex numbers supported on integers coprime to $q$ and any integer
$a$ coprime to $q$, let
$$
W = \sum_{k\le K}\sum_{\ell\le L}
\alpha_k\,\beta_\ell\,\chi(k\ell+a).
$$
Then
$$
W \le AB\left (KL^{1/2}+q^{1/4}K^{1/2}L+\frac{KL}{q^{1/4}}\right)q^{o(1)},
$$
where
$$
A=\max_{k\le K}|\alpha_k| \mand B=\max_{\ell\le L}|\beta_\ell|.
$$
\end{lemma}
\begin{proof}
By the Cauchy-Schwartz inequality
\begin{align*}
|W|^2&\le A^2K\sum_{k\le K}\left|\sum_{\ell \le L}\beta_{\ell}\chi(k\ell+a)\right|^2 \\
&\le A^2B^2K^2L+ \left| \sum_{k\le K}\sum_{\substack{\ell_1, \ell_2 \le L \\ \ell_1 \neq \ell_2}}\beta_{\ell_1}\overline\beta_{\ell_2}\chi(k\ell_1+a)
\overline\chi(k\ell_2+a)\right|.
\end{align*}
Let 
$$W_1=\sum_{k\le K}\sum_{\substack{\ell_1, \ell_2 \le L \\ \ell_1 \neq \ell_2}}\beta_{\ell_1}\overline\beta_{\ell_2}\chi(k\ell_1+a)
\overline\chi(k\ell_2+a),$$
then we have
\begin{align*}
|W_1| &\le  \frac{B^2}{q}\sum_{\substack{\ell_1<\ell_2\le L \\ (\ell_1,q)=1 \\ (\ell_2,q)=1}}\left| \sum_{s=1}^{q}\sum_{k\le K}e^{-2\pi i s k/q}\sum_{\lambda=1}^{q}\chi(\lambda+a\ell_1^{-1})
\overline\chi(\lambda +a\ell_2^{-1})e^{2\pi i s \lambda/q} \right| \\
&\le  \frac{B^2}{q}\sum_{\substack{\ell_1<\ell_2\le L \\ (\ell_1,q)=1 \\ (\ell_2,q)=1}}\sum_{s=1}^{q} \left|\sum_{k\le K}e^{-2\pi i s k/q}\right|\left|\sum_{\lambda=1}^{q}\chi(\lambda+a\ell_1^{-1})
\overline\chi(\lambda +a\ell_2^{-1})e^{2\pi i s \lambda/q} \right|.
\end{align*}
By Lemma~\ref{character sum bound 11},
\begin{align*}
& \sum_{\substack{\ell_1<\ell_2\le L \\ (\ell_1,q)=1 \\ (\ell_2,q)=1}}\sum_{s=1}^{q} \left|\sum_{k\le K}e^{-2\pi i s k/q}\right|\left|\sum_{\lambda=1}^{q}\chi(\lambda+a\ell_1^{-1})
\overline\chi(\lambda +a\ell_2^{-1})e^{2\pi i s \lambda/q} \right|\ll \\
&  \sum_{\ell_1<\ell_2\le L}\sum_{s=1}^{q}\min{\left(K,\frac{1}{||s/q||}\right)}(\ell_1-\ell_2,q)q^{1/2+o(1)}, \\
\end{align*}
and since
\begin{align*}
\sum_{\substack{\ell_1,\ell_2 \le L \\ \ell_1 \neq \ell_2}}(\ell_1-\ell_2,q)& \ll\sum_{ \ell \le L}\sum_{\substack{\ell_1,\ell_2 \le L \\
\ell_1 < \ell_2 \\ \ell_1-\ell_2=\ell}}(\ell,q)
\le L\sum_{d|q}\sum_{ \substack{\ell \le L \\ d|\ell}}1\le L^2q^{o(1)},
\end{align*}
we get
\begin{align*}
|W_1|&\le \frac{B^2}{q}\left(\sum_{s=1}^{q}\min{\left(K,\frac{1}{||s/q||}\right)}\right)q^{1/2+o(1)}L^2 \le B^2\left(1+\frac{K}{q}\right)q^{1/2+o(1)}L^2,
\end{align*}
so that
\begin{align*}
|W|^2\le A^2B^2K\left(KL+\left(1+\frac{K}{q}\right)q^{1/2+o(1)}L^2\right).
\end{align*}
\end{proof}
Next, we use an idea of Garaev~\cite{Gar} to derive a variant of
Lemma~\ref{lem:DoubleChar2} in which the summation
limits over $\ell$ depend on the parameter $k$.

\begin{lemma}
\label{lem:DoubleCharVar4}
Let $ K, L$ be natural numbers  and
let the  sequences $(L_k)_{k=1}^K$ and $(M_k)_{k=1}^K$
of nonnegative integers be such that $M_k< L_k\le L$ for each $k$.
For any two sequences $(\alpha_k)_{k=1}^K$
and $(\beta_\ell)_{\ell=1}^L$ of complex numbers supported on integers coprime to $q$ and for any integer
$a$ coprime to $q$, let
$$
\widetilde{W} = \sum_{k\le K}\sum_{M_k<\ell\le L_k}
\alpha_k\,\beta_\ell\,\chi(k\ell+a).
$$
Then
$$
\widetilde{W} \ll \left(KL^{1/2}+(1+K^{1/2}q^{-1/2})q^{1/4}K^{1/2}L\right)(Lq)^{o(1)},
$$

where
$$
A=\max_{k\le K}|\alpha_k|\mand B=\max_{\ell\le L}|\beta_\ell|.
$$
\end{lemma}

\begin{proof}
For real $z$ we denote
$$
\e_L(z)=\exp(2\pi i z/L).
$$
For each inner sum, using the orthogonality of exponential functions, we have
\begin{equation*}
\begin{split}
\sum_{M_k<\ell\le L_k}\beta_\ell\,\chi(k\ell+a)
& =\sum_{\ell \le L}\sum_{M_K<s\le L_k} \beta_\ell\,\chi(k\ell+a)\cdot\frac{1}{L} \sum_{-\frac12 L<r\le\frac12 L}
\e_L(r(\ell-s))\\
& =\frac{1}{L}\sum_{-\frac12 L<r\le\frac12 L} \sum_{M_k< s \le L_k} \e_L(-r s)\sum_{\ell\le L}
\beta_\ell\,\e_L(r \ell)\,\chi(k\ell+a).
\end{split}
\end{equation*}
In view of~\cite[Bound~(8.6)]{IwKow}, for each $k\le K$
and every integer $r$ such that $|r|\le \tfrac12 L$ we can write
$$
\sum_{M_k<s\le L_k}\e_L(-rs)
=\sum_{s\le L_k}\e_L(-rs)
-\sum_{s\le M_k}\e_L(-rs)
=\eta_{k,r}\frac{L}{|r|+1},
$$
for some complex number $\eta_{k,r}\ll 1$. Thus, if we put
$\widetilde\alpha_{k,r}=\alpha_k\,\eta_{k,r}$ and
$\widetilde\beta_{\ell,r}=\beta_\ell\,\e_L(r \ell)$,
it follows that
$$
\sum_{K_0<k\le K}\sum_{M_k<\ell\le L_k}\alpha_k\,\beta_\ell\,\chi(k\ell+a)
=\sum_{-\frac12 L<r\le\frac12 L}\frac{1}{|r|+1}
\sum_{k\le K}\sum_{\ell\le L}\widetilde\alpha_{k,r}
\widetilde\beta_{\ell,r}\,\chi(k\ell+a).
$$
Applying Lemma~\ref{lem:DoubleChar2} with the sequences
$(\widetilde\alpha_{k,r})_{k=1}^K$ and $(\widetilde\beta_{\ell,r})_{\ell=1}^L$,
and noting that
$$
\sum_{-\frac12 L<r\le\frac12 L}\frac{1}{|r|+1}\ll\log  L,
$$
we derive the stated bound.
\end{proof}

\section{Proof of Theorem~\ref{thm:main}}
Considering the sum
$$S_a(q;N)=\sum_{n\le N}\Lambda(n)\chi(n+a),$$
we apply Lemma~\ref{lem:Vau}  with 
$$f(n)=\begin{cases} \chi(n+a) \quad \text{if} \ (n,q)=1 \\
0 \quad \text{otherwise}. 
\end{cases}
$$
For $\Sigma_1$ in Lemma~\ref{lem:Vau} we apply the trivial estimate,
$$\Sigma_1 = \left|\sum_{n\le U}\Lambda(n)f(n)\right| \ll U.$$
\subsection{The sum $\Sigma_2$}
We have
$$\Sigma_2  =  (\log UV)\sum_{\substack{v \le UV \\ (v,q)=1}}\left|\sum_{\substack{s\le N/v \\ (s,q)=1}}\chi(sv+a)\right|
= (\log UV)\sum_{\substack{v \le UV \\ (v,q)=1}}\left|\sum_{\substack{s\le N/v \\ (s,q)=1}}\chi(s+av^{-1})\right|.
$$
 By Lemma~\ref{polya vinogradov}, since $N\le q$
\begin{align*}
\sum_{\substack{v \le Nq^{-43/72} \\ (v,q)=1}}\left|\sum_{\substack{s\le N/v \\ (s,q)=1}}\chi(s+av^{-1})\right| &\le
\sum_{\substack{v \le Nq^{-43/72} \\ (v,q)=1}}q^{1/2+o(1)} \\ &\le Nq^{-7/72+o(1)}
 ,\end{align*}
by Lemma~\ref{burgess 2 1}
\begin{align*}
\sum_{\substack{ Nq^{-43/72}<v \le Nq^{-11/24} \\ (v,q)=1}}\left|\sum_{\substack{s\le N/v \\ (s,q)=1}}\chi(s+av^{-1})\right|&
\le q^{3/16+o(1)}N^{1/2}\left(\sum_{\substack{ Nq^{-43/72}<v \le Nq^{-11/24}}}v^{-1/2}\right) \\ &\le Nq^{-1/24+o(1)},
\end{align*}
and by Lemma~\ref{burgess 3 1}
\begin{align*}
\sum_{\substack{ Nq^{-11/24}<v \le UV \\ (v,q)=1}}\left|\sum_{\substack{s\le N/v \\ (s,q)=1}}\chi(s+av^{-1})\right|&
\le q^{1/9+o(1)}N^{2/3}\left(\sum_{\substack{ Nq^{-11/24}<v \le UV}}v^{-2/3} \right) \\
&\le q^{1/9+o(1)}N^{2/3}(UV)^{1/3}.
\end{align*}
Combining the above bounds gives
$$\Sigma_2\le \left(Nq^{-1/24+o(1)}+q^{1/9}N^{2/3}(UV)^{1/3}\right)q^{o(1)}.$$
\subsection{The sum $\Sigma_3$}
As above, we get
$$ \Sigma_3= (\log N)\sum_{\substack{v\le V \\ (v,q)=1}}\,\max_{w\ge 1}
\left|\sum_{\substack{w\le s\le N/v \\ (s,q)=1}}\chi(s+av^{-1})\right|\le \left(Nq^{-1/24+o(1)}+q^{1/9}N^{2/3}V^{1/3}\right)q^{o(1)}.$$
\subsection{The sum $\Sigma_4$}
For the sum $\Sigma_4$, we have
$$
\Sigma_4 = \sum_{\substack{U<k\le\frac{N}{V} \\
\gcd(k,\,q)=1}}\Lambda(k)\sum_{\substack{V<\ell\le N/k}}A(\ell)\chi(k\ell+a),
$$
where
$$
A(\ell)=\sum_{d|\ell,\,d\le V}\mu(d), \quad \gcd(\ell, q)=1,
$$
and
$$
A(\ell) = 0, \quad \gcd(\ell, q)>1.
$$
Note that
$$
\Lambda(k)\le\log k\le k^{o(1)} \mand |A(\ell)|\le\tau(\ell)\le\ell^{o(1)}.
$$

We separate the sum  $\Sigma_{4}$ into $O(\log N)$ sums of the form
$$
W(K) = \sum_{\substack{K<k\le 2K \\
\gcd(k,\,q)=1}}\Lambda(k)\sum_{\substack{V<\ell\le N/k}}A(\ell)\chi(k\ell+a),
$$
where $U \le K \le N/V$. By Lemma~\ref{lem:DoubleCharVar4} we have
\begin{equation}
\label{eq: WK extra large K}
\begin{split}
W(K)&\le \left(K^{1/2}N^{1/2}+(1+K^{1/2}q^{-1/2})q^{1/4}K^{-1/2}N\right)q^{o(1)} \\
&\le K^{1/2}N^{1/2}+q^{1/4}K^{-1/2}N+Nq^{-1/4+o(1)},
\end{split}
\end{equation}
so that summing over the $O(\log{N})$ values of $U\le K \le NV^{-1}$ gives
$$\Sigma_4\le \left(NV^{-1/2}+q^{1/4}NU^{-1/2}+Nq^{-1/4}\right)(Nq)^{o(1)}.$$
\subsection{Concluding the Proof}
Combining the estimates for $\Sigma_1, \dots, \Sigma_4$ gives
\begin{align*}
S_a(q;N)\le \left(Nq^{-1/24+o(1)}+U+NV^{-1/2}+q^{1/4}NU^{-1/2}+q^{1/9}N^{2/3}(UV)^{1/3}\right)(Nq)^{o(1)}.
\end{align*}
We choose $U=q^{1/2}V$ to balance the terms $NV^{-1/2}$ and $q^{1/4}NU^{-1/2}$ which gives
$$S_a(q;N)\le \left(Nq^{-1/24+o(1)}+U+NV^{-1/2}+q^{5/18}N^{2/3}V^{2/3}\right)(Nq)^{o(1)}.$$
Choosing $V=N^{2/7}q^{-5/21}$ to balance the terms $NV^{-1/2}$ and $q^{5/18}N^{2/3}V^{2/3}$ we get
$$S_a(q;N)\le \left(Nq^{-1/24+o(1)}+q^{11/42}N^{2/7}+q^{5/42}N^{6/7}\right)(Nq)^{o(1)}.$$
We have $U\ge V \ge 1$ when $N\ge q^{5/6}$, which is when the term $q^{5/42}N^{6/7}$ becomes nontrivial. Also we need
$$UV=q^{1/42}N^{4/7}\le N$$
which is satisfied for $N\ge q^{1/18}$ which we may suppose since otherwise the bound is trivial. Finally we note that we may remove the middle term, since it is dominated by the last term for $N\ge q^{1/4}$.

\end{document}